\documentclass[12pt]{amsart}

\usepackage{psfrag}
\usepackage{color}
\usepackage{tikz}
\usetikzlibrary{matrix,arrows}
\usepackage{graphicx,graphics}
\usepackage{fullpage,amssymb,amsfonts,amsmath,amstext,amsthm,amscd,verbatim,enumerate}
\usepackage[T1]{fontenc}

\begin{document}

\newtheorem{theorem}{Theorem}[section]
\newtheorem{result}[theorem]{Result}
\newtheorem{fact}[theorem]{Fact}
\newtheorem{example}[theorem]{Example}
\newtheorem{conjecture}[theorem]{Conjecture}
\newtheorem{lemma}[theorem]{Lemma}
\newtheorem{proposition}[theorem]{Proposition}
\newtheorem{corollary}[theorem]{Corollary}
\newtheorem{facts}[theorem]{Facts}
\newtheorem{props}[theorem]{Properties}
\newtheorem*{thmA}{Theorem A}
\newtheorem{ex}[theorem]{Example}
\theoremstyle{definition}
\newtheorem{definition}[theorem]{Definition}
\newtheorem{remark}[theorem]{Remark}
\newtheorem*{defna}{Definition}

\newcommand{\notes} {\noindent \textbf{Notes.  }}
\newcommand{\note} {\noindent \textbf{Note.  }}
\newcommand{\defn} {\noindent \textbf{Definition.  }}
\newcommand{\defns} {\noindent \textbf{Definitions.  }}
\newcommand{\x}{{\bf x}}
\newcommand{\z}{{\bf z}}
\newcommand{\B}{{\bf b}}
\newcommand{\V}{{\bf v}}
\newcommand{\T}{\mathbb{T}}
\newcommand{\Z}{\mathbb{Z}}
\newcommand{\Hp}{\mathbb{H}}
\newcommand{\D}{\mathbb{D}}
\newcommand{\R}{\mathbb{R}}
\newcommand{\N}{\mathbb{N}}
\renewcommand{\B}{\mathbb{B}}
\newcommand{\C}{\mathbb{C}}
\newcommand{\ft}{\widetilde{f}}
\newcommand{\dt}{{\mathrm{det }\;}}
 \newcommand{\adj}{{\mathrm{adj}\;}}
 \newcommand{\0}{{\bf O}}
 \newcommand{\av}{\arrowvert}
 \newcommand{\zbar}{\overline{z}}
 \newcommand{\xbar}{\overline{X}}
 \newcommand{\htt}{\widetilde{h}}
\newcommand{\ty}{\mathcal{T}}
\renewcommand\Re{\operatorname{Re}}
\renewcommand\Im{\operatorname{Im}}
\newcommand{\tr}{\operatorname{Tr}}

\newcommand{\ds}{\displaystyle}
\numberwithin{equation}{section}

\renewcommand{\theenumi}{(\roman{enumi})}
\renewcommand{\labelenumi}{\theenumi}

\title{Fixed curves near fixed points}

\author{Alastair Fletcher}
\address{Department of Mathematical Sciences, Northern Illinois University, DeKalb, IL 60115-2888. USA}
\email{fletcher@math.niu.edu}

\maketitle

\begin{abstract}
Let $H$ be a composition of an $\R$-linear planar mapping and $z\mapsto z^n$. We classify the dynamics of $H$ in terms of the parameters of the $\R$-linear mapping and the degree by associating a certain finite Blaschke product. We apply this classification to this situation where
$z_0$ is a fixed point of a planar quasiregular mapping with constant complex dilatation in a neighbourhood of $z_0$. In particular we find how many curves there are that are fixed by $f$ and that land at $z_0$.
\end{abstract}

\section{Introduction}

\subsection{Background}

Complex dynamics has been a field of intense study over the last thirty years. The striking computer generated images of the Mandelbrot set helped inspire this surge of activity and showed how very complicated behaviour can arise from very simply defined iterative systems. Yet complex dynamics had its first burst of interest at the end of the nineteenth and into the beginning of the twentieth centuries. Koenigs and B\"ottcher classified the behaviour of holomorphic functions near fixed points by conjugating to simpler functions. In a neighbourhood of a fixed point, a holomorphic function can be conjugated to either $z\mapsto \lambda z$ or $z\mapsto z^n$ depending on whether or not the holomorphic function is injective in a neighbourhood of the fixed point. See Milnor's book \cite{Milnor} for an exposition of these ideas.

Quasiconformal mappings and quasiregular mappings provide natural higher dimensional analogues for holomorphic functions in the plane. Informally speaking, quasiregular mappings are mappings which allow a bounded amount of distortion. They share many value distributional properties with holomorphic functions, for example versions of Picard's and Montel's Theorems hold, but they are more flexible than holomorphic functions. The only holomorphic functions in $\R^n$ for $n\geq 3$ are M\"obius transformations, and so it is natural to allow distortion to have an interesting function theory. See Rickman's book \cite{Rickman} for an introduction to the theory of quasiregular mappings.

Much more is known about quasiregular mappings in the planar setting: every quasiregular mapping has a {\it Stoilow decomposition}, that is, it can be written as a composition of a holomorphic function and a quasiconformal mapping. The holomorphic part takes care of the branching and the quasiconformal part takes care of the distortion. The main reason that more is known in the plane is due to the Measurable Riemann Mapping Theorem, which states that solutions of the Beltrami differential equation $f_{\overline{z}} = \mu f_z$ for $\mu \in L^{\infty}(\C)$ with $||\mu ||_{\infty} <1$ are quasiconformal mappings which can be assumed to fix $0,1$ and $\infty$. Conversely, given a quasiconformal mapping $f:\C \to \C$, its {\it complex dilatation} $\mu_f = f_{\overline{z}}/f_z$ is contained in the unit ball of $L^{\infty}(\C)$.

The Measurable Riemann Mapping Theorem was used by Douady and Hubbard, and Sullivan to great effect in the 1980s in proving fundamental results in complex dynamics. See Branner and Fagella's book \cite{BF} for a survey of the use of quasiconformal and quasiregular methods in modern complex dynamics.

More recently, there has been an interest in studying the iteration of quasiregular mappings themselves.
A composition of quasiregular mappings is again quasiregular.
However, the distortion of the iterates of a quasiregular mapping will typically increase. This means the machinery available from Montel's Theorem is unavailable to help with the iterative theory. Remarkably, it is still possible to say quite a lot about the dynamics of quasiregular mappings when there is not a uniform bound on the distortion of the iterates. See for example the recent works of Bergweiler \cite{B1,B2}.

\subsection{Overview of the paper}

In section 2, we will recall the definition of quasiconformal and quasiregular mappings, with focus on the planar case.
Taking as an inspiration the work of Douady and Hubbard in studying the simplest non-trivial polynomials $z^2+c$, an initial step for studying quasiregular mappings in the plane is to study those which are the simplest: using the Stoilow decompostion, these are mappings which can be written as a composition of a power of $z$ and a quasiconformal mapping of constant complex dilatation. These were first studied in \cite{FG} and in more detail for the degree two case in \cite{FF}. We remark that similar quasiregular perturbations of polynomials were studied in \cite{BSTV,BP,BvN,Peckham,PM}. 

By construction, every such mapping maps rays emanating from $0$ onto rays and so every such mapping induces a degree $n$ circle endomorphism. The major insight here is that such a circle map is strongly related to a particular Blaschke product and the dynamics of the Blaschke product has strong implications for the dynamics of the original quasiregular mapping.

In particular, every fixed point of the Blaschke product corresponds to either a fixed ray of the quasiregular mapping or a pair of opposite rays which switch. Here, the cases of even or odd $n$ are different. Further, the Julia set of a Blaschke product is either the whole circle or a Cantor subset of it. We classify the type of the quasiregular mapping in terms of the complex dilatation by relating it to the parameter space of degree $n$ unicritical Blaschke products.

We next show that for such quasiregular mappings, the plane breaks into three dynamically interesting sets: the basins of attraction of $0$ and $\infty$ respectively and the boundary between them. We show that as long as the distortion is smaller than the degree, the boundary is the Julia set of the quasiregular mapping. Otherwise there are cases when the boundary of basins of attraction fails to have the necessary blowing-up property required for the Julia set.

An important point in the analysis of these mappings is that they are not uniformly quasiregular. For otherwise they would just be quasiconformal conjugates of holomorphic mappings and this study would not be of independent interest. We will show that in fact the distortion of the iterates blows up at every point for such mappings.

Finally, we show how to construct a B\"ottcher type coordinate for a fixed point of a quasiregular mapping for which the complex dilatation is constant in a neighbourhood of the fixed point. This allows the results from the rest of the paper to be applied locally.
See \cite{FF1} for the degree $2$ case of such a B\"ottcher coordinate.\\

The author would like to thank Doug Macclure for producing Figures \ref{fig:3} and \ref{fig:4} used in this paper.

\section{Quasiregular mappings with constant complex dilatation}

\subsection{Quasiregular mappings}

We first recall the definitions of quasiconformal and quasiregular mappings in the plane.

A \emph{quasiconformal} mapping $f:\C \to \C$ is a homeomorphism so that $f$ is in the Sobolev space $W^1_{2, loc}(\C)$ and there exists $k\in[0,1)$ such that
the \emph{complex dilatation} $\mu_f = f_{\zbar}/ f_z$ satisfies
\[ | \mu_f(z) | \leq k <1\]
almost everywhere in $\C$. See for example \cite{FM} for more details on quasiconformal mappings.
The \emph{distortion} of $f$ at $z\in\C$ is
\[K_f(z):=\frac{1+|\mu_f(z)|}{1-|\mu_f(z)|}.\]
A mapping is called $K$-quasiconformal if $K_f(z) \leq K$ almost everywhere. The smallest such constant is called the \emph{maximal dilatation} and denoted by $K_f$. The case $K_f = 1$ corresponds to biholomorphic mappings and so $K_f$ is a way of measuring how far from a conformal mapping $f$ is.

If we drop the assumption on injectivity, then $f$ is a \emph{quasiregular mapping}. See for example \cite{IM,Rickman} for the theory of quasiregular mappings. Every quasiregular mapping is locally quasiconformal away from the branch set, which in the planar case is discrete. We can therefore consider the complex dilatation of a quasiregular mapping.
In the plane, every quasiregular mapping has an important factorization.

\begin{theorem}[Stoilow factorization, see for example \cite{IM} p.254]
\label{Stoilow}
Let $f:\C \rightarrow \C$ be a quasiregular mapping. Then there exists an holomorphic function $g$ and a quasiconformal mapping $h$ such that $f = g \circ h$.
\end{theorem}

In this decomposition, the holomorphic part takes care of the branching and the quasiconformal part deals with the distortion.

We call a quasiregular mapping $f$ \emph{uniformly quasiregular} if there exists $K\geq1$ such that
$K_{f^n}(z)\leq K$ for all $n\in\N$.
The dynamics of uniformly quasiregular mappings in the plane are well understood due to results of Hinkkanen \cite{Hinkkanen} and Sullivan \cite{Sullivan} that state that every uniformly quasiregular map $f:\C\to\C$ is quasiconformally conjugate to a holomorphic map.

\subsection{$\R$-linear mappings}

Let $K>1$ and $\theta \in (-\pi/2,\pi/2]$. Denote by $h=h_{K,\theta}$ the $\R$-linear mapping
\begin{equation} 
\label{eq:h}
h_{K,\theta}(z) = \left( \frac{K+1}{2} \right ) z +e^{2i\theta} \left ( \frac{K-1}{2} \right ) \overline{z}.
\end{equation}
This mapping stretches by a factor $K$ in the direction $e^{i\theta}$. This is a quasiconformal mapping and its complex dilatation is the constant
\[ \mu_h(z) \equiv e^{2i\theta}\left ( \frac{K-1}{K+1} \right ).\]
Every quasiconformal mapping $h:\C \to \C$ with constant complex dilatation can be written as $h=A\circ h_{K,\theta}$ for some $K,\theta$ and where $A$ is a M\"obius transformation (see \cite[Proposition 1.1]{FF}).

\subsection{Quasiregular mappings with constant complex dilatation}

Let $n\in \N$ with $n\geq 2$ and define $H = H_{K,\theta,n}$ by
\begin{equation}\label{eq:H}
H(z) = [h_{K,\theta}(z)]^n.
\end{equation}
This is the mapping whose dynamics will be studied in this paper. We refer to \cite{FF} for an analysis of the $n=2$ case. Here, however, we will give a unified treatment for all $n\geq 2$ and classify the behaviour of mappings given by \eqref{eq:H} into  classes which depend on $K,\theta$ and $n$. 

A ray is a semi-infinite line of the form $R_{\phi} = \{z\in \C : \arg(z) = \phi \}$. Since $h$ and $z^n$ map rays to rays, so does $H$. Since $H$ maps rays to rays, it induces a degree $n$ circle endomorphism that we will denote by $\widetilde{H}$.

\begin{proposition}
\label{prop:hform}
Given $H$ as in \eqref{eq:H}, we may write
\[ H(re^{i\phi}) = A(r,\phi)\widetilde{H}(e^{i\phi}),\]
where $A(r,\phi) = r^n(1+ (K^2-1)\cos^2(\phi - \theta) )^{n/2}$ and
\begin{equation} 
\label{eq:polarH} 
\tan \left ( \frac{\arg \widetilde{H}(e^{i\phi})}{n} - \theta \right ) = \frac{ \tan(\phi - \theta)}{K}.
\end{equation}
\end{proposition}

\begin{proof}
This is an elementary exercise, and so we just sketch the details. Since $h$ has the form \eqref{eq:h}, then writing $z=re^{i\phi}$ yields
\begin{align*}
h(re^{i\phi}) &= \left (\frac{K+1}{2} \right ) re^{i\phi} + e^{2i\theta}\left ( \frac{K-1}{2} \right ) re^{-i\phi}\\
&=\left (\frac{K+1}{2} \right ) r\cos \phi + \left (\frac{K-1}{2} \right )r\cos(\phi - 2\theta) + i\left [ \left (\frac{K+1}{2} \right ) r\sin \phi  + \left (\frac{K-1}{2} \right ) r\sin (\phi - 2\theta) \right ].
\end{align*}
Therefore
\begin{align*}
|h(re^{i\phi})|^2 &= \left (\frac{K+1}{2} \right )^2r^2 + \frac{(K^2-1)}{2} r^2 \cos (2(\phi - \theta)) + \left (\frac{K-1}{2} \right )^2r^2 \\
&=r^2( 1+ (K^2-1)\cos^2(\phi - \theta) ).
\end{align*}
Since $H(z) = h(z)^n$, we obtain the formula for $A$.

Next, let $h_0(x+iy) = Kx+iy$ for $K>1$. Then
\[ \tan \arg h_0(re^{i\phi}) = \frac{\tan(\phi)}{K}.\]
In general, $h_{K,\theta}(z) = e^{i\theta}(h_0(e^{-i\theta}z))$ and so
\[ \tan \arg( h(re^{i\phi}) - \theta) = \frac{\tan(\phi - \theta)}{K}.\]
Finally, since $H(z) = h(z)^n$, we obtain \eqref{eq:polarH}.
\end{proof}

\section{Circle endomorphisms and Blaschke products}

In this section, we will show how $\widetilde{H}$ is related to a Blaschke product. First of all, since $\widetilde{H}$ is an orientation preserving degree $n$ mapping of $\partial \D$ it is a circle endomorphism.
Every circle endomorphism $g$ of degree $n$ can be lifted to a mapping $\widehat{g}:\R \to \R$ which satisfies
\[ g(x+2\pi) = g(x) + 2\pi n.\]

\begin{definition}
\label{def:tg}
Given a circle endomorphism $g:\partial \D \to \partial \D$ of degree $n$, consider its lift $\widehat{g}$ to $\R$.
Then we define $T(g)$ to be the degree $n$ endomorphism whose lift to $\R$ is given in $[0,2\pi)$ by
\[ \widehat{T(g)}(x) = \frac{ \widehat{g}(2x)}{2}.\]
\end{definition}

The map $T(g)$ is well-defined and essentially rescales $g$ by a factor $2$. We can view $T(g)$ informally as the conjugate of $g$ by $z^2$, since $g(z^2) = [T(g)(z)]^2$ by construction. See Figure \ref{fig:1} for a diagram showing how $g$ and $T(g)$ are related on $\R / (2\pi \Z)$. 

\begin{figure}[h]
\begin{center}
\includegraphics[width = 7in]{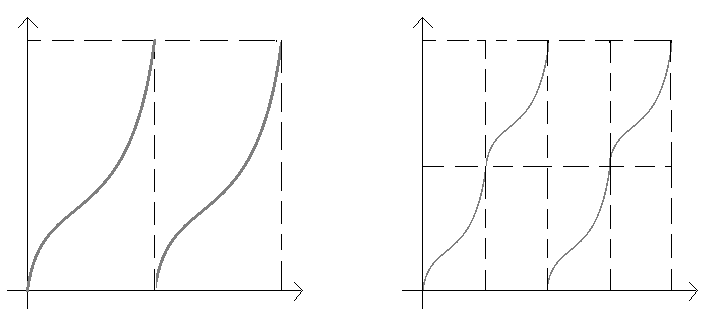}
\caption{A lift of a degree $2$ map $g$ to $[0,2\pi)$ on the left, and the lift of the corresponding map $T(g)$ on the right.}
\label{fig:1}
\end{center}
\end{figure}

\begin{theorem}
\label{thm:agree}
Let $H = H_{K,\theta,n}$ for $K>1$, $\theta \in (-\pi/2, \pi/2]$ and $n\geq 2$.
The map $\widetilde{H}:\partial \D \to \partial \D$ agrees with the function $T(B)$, where $B$ is the Blaschke product defined by
\begin{equation}\label{eq:B} B(z) = \left ( \frac{z+\mu}{1+\overline{\mu}z} \right ) ^n,\quad |z|=1,\end{equation}
and where $\mu = e^{2i\theta} \left (\frac{K-1}{K+1} \right )$.
\end{theorem}

\subsection{Blaschke products}

Before proving Theorem \ref{thm:agree}, we will review some of the properties of Blaschke products and, in particular, dynamical aspects we will need.

Recall that a finite Blaschke product is a function $B:\overline{\C} \to \overline{\C}$ given by
\begin{equation}\label{eq:bl} B(z) = e^{i\theta} \prod_{i=1}^n \left ( \frac{z-w_i}{1-\overline{w_i}z }\right ),\end{equation}
for some $\theta \in [0,2\pi)$ and $w_i \in \D$ for $i=1,\ldots,n$. We call a Blaschke product non-trivial if it is not a M\"obius mapping, that is, if $n\geq 2$. For a finite Blaschke product, $\D, \partial \D$ and $\overline{\C} \setminus \overline{\D}$ are all completely invariant.

Every finite degree self-mapping of $\D$ is a finite Blaschke product \cite[p.19]{BM}, and so they can be viewed as analogues for polynomials in the disk. 
By the Schwarz-Pick Lemma, $B$ can have at most one fixed point in $\D$. If $z_0$ is a fixed point of $B$, then it is straightforward to show that $1/\overline{z_0}$ is also a fixed point of $B$. Hence all but possibly two (with the convention that infinity is a fixed point if some $w_i=0$) of the fixed points of $f$ must lie on $\partial \D$. 

The Denjoy-Wolff Theorem \cite[p.58]{Milnor} states that if $f:\D\to \D$ is holomorphic and not an elliptic M\"obius mapping then there is some point $z_0 \in \overline{\D}$ such that $f^n(z) \to z_0$ for every $z\in \D$. We call such a point a {\it Denjoy-Wolff point} of $f$.

Using the Denjoy-Wolff Theorem, there is a classification of finite Blaschke products in analogy with that for M\"obius transformations:
\begin{enumerate}[(i)]
\item $B$ is called {\it hyperbolic} if the Denjoy-Wolff point $z_0$ of $B$ lies on $\partial \D$ and $B'(z_0)<1$,
\item $B$ is called {\it parabolic} if the Denjoy-Wolff point $z_0$ of $B$ lies on $\partial \D$ and $B'(z_0)=1$,
\item $B$ is called {\it elliptic} if the Denjoy-Wolff point $z_0$ of $B$ lies in $\D$. In this case, we must have $|B'(z_0)|<1$.
\end{enumerate}

It is not hard to see that the Julia set of a Blaschke product must be contained in $\partial \D$ and is either the whole of $\partial \D$ or a Cantor subset of $\partial \D$. We summarize the classification of Blaschke products in terms of the Julia set as follows, see for example \cite{CDP}.

\begin{theorem}[\cite{CDP}]
Let $B$ be a non-trivial finite Blaschke product. Then $J(B) = \partial \D$ if and only if $B$ is elliptic or $B$ is parabolic and $B''(z_0) =0$, where $z_0$ is the Denjoy-Wolff point of $B$ on $\partial \D$. On the other hand, $J(B)$ is a Cantor subset of $\partial \D$ if and only if $B$ is hyperbolic or $B$ is parabolic and $B''(z_0)\neq 0$.
\end{theorem}

The class of Blaschke products we will be interested in in this paper are the unicritical Blaschke products, namely those with one critical point in $\D$.
For $n\geq 2$, we define the set 
\[ \mathcal{B}_n  = \left \{ B_w(z):= \left ( \frac{z-w}{1-\overline{w}z} \right )^n : w\in \D, \arg(w) \in \left [0,\frac{2\pi}{n-1} \right ) \right \}\]
of normalized unicritical Blaschke products of degree $n$, which is parameterized by the sector 
\[ S_n = \left \{ w\in \D: \arg(w) \in \left [0,\frac{2\pi}{n-1} \right ) \right \} .\]
We denote by $\mathcal{E}_n$ those parameters in $S_n$ which give elliptic unicritical Blaschke products and by $\mathcal{M}_n$ those parameters in $S_n$ which give rise to unicritical Blaschke products with connected Julia set.
We further denote by $\widetilde{\mathcal{E}_n}\subset \D$  the set 
\[ \widetilde{\mathcal{E}_n} = \bigcup _{j=0}^{n-2} R_j(\mathcal{E}_n),\]
where $R_j$ is the rotation through angle $2\pi j/(n-1)$, and denote by $\widetilde{\mathcal{M}_n}$ the corresponding set for $\mathcal{M}_n$. The following theorem summarizes results in \cite{Fletcher}.

\begin{theorem}[\cite{Fletcher}]
\label{thm:F1}
Let $n\geq 2$. 
\begin{enumerate}[(i)]
\item Every unicritical Blaschke product of degree $n$ is conjugate by M\"obius mappings to a unique element of $\mathcal{B}_n$.
\item The connectedness locus $\mathcal{M}_n$ consists of $\mathcal{E}_n$ and one point on the relative boundary in $S_n$ where $|w| = \frac{n-1}{n+1}$.
\item The set $\widetilde{\mathcal{E}_n} \subset \D$ is a starlike domain about $0$ which contains the disk $ \{w\in \D : |w| <\frac{n-1}{n+1} \}$. 
\item The set $\widetilde{\mathcal{M}_n}$ consists of $\widetilde{\mathcal{E}_n}$ and $n-1$ points on its relative boundary in $\D$.
\item If $n$ is even, then the ray $\{ re^{i\phi} :0\leq r <1 \}$ is contained in $\widetilde{\mathcal{E}_n}$ for $\phi = 2k\pi/(n-1)$ where $k=0,1,\ldots,n-2$. The ray $\{ re^{i\phi} :0\leq r \leq \frac{n-1}{n+1} \}$ is contained in $\widetilde{\mathcal{M}_n}$ for $\phi = (2k+1)\pi/(n-1)$ where $k=0,1,\ldots,n-2$.
\item If $n$ is odd, then the ray $\{ re^{i\phi} :0\leq r <1 \}$ is contained in $\widetilde{\mathcal{E}_n}$ for $\phi = (2k+1)\pi/(n-1)$ where $k=0,1,\ldots,n-2$. The ray $\{ re^{i\phi} :0\leq r \leq \frac{n-1}{n+1} \}$ is contained in $\widetilde{\mathcal{M}_n}$ for $\phi = 2k\pi/(n-1)$ where $k=0,1,\ldots,n-2$.
\end{enumerate}
\end{theorem}

We remark that the boundary curve of $\widetilde{\mathcal{E}_n}$ has recently been shown to be an epicycloid in $\overline{\D}$ with $n-1$ cusps, see \cite{CFY}.

\subsection{Connection between $\widetilde{H}$ and Blaschke products}

\begin{proof}[Proof of Theorem \ref{thm:agree}]

Recall from \eqref{eq:polarH} that $\tan \left ( \frac{\arg \widetilde{H}(e^{i\phi})}{n} - \theta \right ) = \frac{ \tan(\phi - \theta)}{K}$. Writing this using the exponential function yields
\[ \frac{ \exp (2i(\arg \widetilde{H}(e^{i\phi})/n - \theta )) -1 }{ \exp ( 2i (\arg \widetilde{H}(e^{i\phi})/n - \theta )) +1 }
 = \frac{ \exp (2i(\phi - \theta) )-1}{K (\exp ( 2i(\phi - \theta ) ) +1 )}.\]
Rearranging this in terms of $e^{2i\arg \widetilde{H}(e^{i\phi}) /n}$, we obtain
\[ e^{2i\arg \widetilde{H}(e^{i\phi}) /n} = \frac{ e^{2i\phi} + e^{2i\theta} \left( \frac{K-1}{K+1} \right ) }{1+e^{2i\phi} e^{-2i\theta} \left ( \frac{K-1}{K+1} \right ) }  = \frac{ e^{2i\phi} + \mu }{1+\overline{\mu} e^{2i\phi} },\]
where $\mu = e^{2i\theta} \left ( \frac{K-1}{K+1} \right )$. We therefore see that
\[ \arg \widetilde{H}(e^{i\phi}) = \frac{n}{2} \arg \left ( \frac{e^{2i\phi} + \mu}{1+\overline{\mu}e^{2i\phi}} \right ).\]
Therefore $\widetilde{H}$ is obtained by taking the Blaschke product
\[ B(z) = \left ( \frac{ z+\mu}{1+\overline{\mu}z} \right ) ^n \]
restricted to $\partial \D$, lifting this circle endomorphism to $\R$, conjugating by $x\mapsto 2x$ and projecting back to $\partial \D$. In summary, $\widetilde{H} = T(B)$ as claimed.
\end{proof}

We remark here that the cases where $n$ is even or odd differ. When $n$ is even, $T(B)$ is the Blaschke product
\[ T(B)(z) = \left ( \frac{z^2+\mu }{1+\overline{\mu}z^2} \right )^{n/2}, \quad |z|=1.\]
However, when $n$ is odd, $T(B)$ is no longer a Blaschke product because $n/2$ is not an integer.

\subsection{Classification of $H$ and fixed rays}

We can classify the mappings $H$ in terms of the associated Blaschke product $B$.

\begin{definition}
Let $n \geq 2$, $K\geq 1$, $\theta \in(-\pi/2,\pi/2]$ and $H = H_{K,\theta,n}$. Denote by $B$ the associated Blaschke product given by \eqref{eq:B} where $\mu = e^{2i\theta}\left ( \frac{K-1}{K+1} \right )$. Then we call $H$ elliptic, parabolic or hyperbolic if $B$ is elliptic, parabolic or hyperbolic respectively.
\end{definition}

Since $H$ maps rays to rays, it is of interest to find which rays are fixed by $H$. Fixed rays of $H$ correspond to fixed points of $\widetilde{H}$ and these have a relation to the fixed points of $B$.

\subsubsection{Even degree}

In this subsection, we assume that $n\geq 2$ is even.
\begin{lemma}
\label{lem:fixpts1}
With the notation above, if $n$ is even then there is a one-to-one correspondence between fixed points of $B$ on $\partial \D$ and fixed points of $\widetilde{H}$. 
\end{lemma}

\begin{proof}
Let $n$ be even. Then $\widetilde{H}(e^{i\phi}) = \widetilde{H}(-e^{i\phi})$. Since 
$\widetilde{H} = T(B)$ by Theorem \ref{thm:agree}, we have
$2\arg \widetilde{H}(e^{i\phi}) = \arg B(e^{2i\phi})$. Suppose $B$ fixes $e^{i\phi_0}$ and $\phi_1 \in \{ \phi_0/2, \phi_0/2 +\pi \}$. Then
\[ 2\arg \widetilde{H}( e^{i\phi_1}) = \arg B(e^{i\phi_0}) = \phi_0,\]
and so $\arg \widetilde{H}( e^{i\phi_1}) \in \{ \phi_1, \phi_1 +\pi \}$. Since $\widetilde{H}$ maps antipodal points onto the same image, this means that one of $e^{i\phi_1}$ and $-e^{i\phi_1}$ is fixed by $\widetilde{H}$ and the other is mapped onto this fixed point. On the other hand, if $e^{i\phi_1}$ is fixed by $\widetilde{H}$, then it is easy to see that $e^{2i\phi_1}$ is fixed by $B$. Hence $B$ and $\widetilde{H}$ have the same number of fixed points.
\end{proof}

\begin{theorem}\label{thm:H1}
Let $n \geq 2$ be even, $K\geq 1$, $\theta \in(-\pi/2,\pi/2]$ and let $\mu = e^{2i\theta} \left (\frac{K-1}{K+1} \right )$. Then
$H=H_{K,\theta,n}$ is elliptic, parabolic or hyperbolic according to whether $-\mu \in \widetilde{\mathcal{E}_n}$, the relative boundary of $\widetilde{\mathcal{E}_n}$ in $\D$ or $\D \setminus \overline{\widetilde{\mathcal{E}_n}}$ respectively. Further, if $n,\theta$ are fixed,
there exists $K_{\theta}\in (1,\infty ]$ such that 
\begin{enumerate}[(i)]
\item for $1\leq K<K_{\theta}$, $H$ is elliptic and $H$ has $n-1$ fixed rays;
\item for $K=K_{\theta}$, $H$ is parabolic and $H$ has at most $n$ fixed rays;
\item for $K>K_{\theta}$, $H$ is hyperbolic and $H$ has $n+1$ fixed rays.
\end{enumerate}
\end{theorem}

\begin{proof}
The first part of this theorem is Theorem \ref{thm:F1} applied to the situation with $H$.
Using the fact that $\widetilde{\mathcal{E}_n}$ is starlike with respect to $0$, if $\theta$ is fixed, then there exists $K_{\theta}\in (1,\infty]$ such that if $K<K_{\theta}$ then $B$ is elliptic, if $K=K_{\theta}$ then $B$ is parabolic, and if $K>K_{\theta}$ then $B$ is hyperbolic.

If $B$ is elliptic, then $B$ has a unique fixed point in $\D$, a unique fixed point in $\overline{\C} \setminus \overline{\D}$ and $n-1$ fixed points on $\partial \D$. 
If $B$ is hyperbolic then it has $n+1$ fixed points on $\partial \D$ and so $H$ has $n+1$ fixed rays. If $B$ is parabolic, then $B$ has at most $n$ fixed points on $\partial \D$. The claims then follow from Lemma \ref{lem:fixpts1}.
\end{proof}

\subsubsection{Odd degree}

In this subsection, we assume that $n\geq 3$ is odd. This case is a little more involved than the even case, because here fixed points of $B$ may not correspond to fixed points of $\widetilde{H}$.

\begin{lemma}
\label{lem:fixpts2}
With the notation as above, if $n$ is odd, then there is a one-to-one correspondence between fixed points of $B$ and pairs of antipodal points on $\partial \D$ which are either both fixed by $\widetilde{H}$ or switched by $\widetilde{H}$. 
Further,
\[ | \{ \text{fixed points of $\widetilde{H}$} \} | +  | \{ \text{switched points of $\widetilde{H}$} \} | = 2 |\{ \text{fixed points of $B$} \} | .\]
\end{lemma}

\begin{proof}
Let $n$ be odd. Then $\widetilde{H}(e^{i\phi}) = -\widetilde{H}(-e^{i\phi})$. As in the proof of Lemma \ref{lem:fixpts1}, if $e^{i\phi_1}$ is fixed by $\widetilde{H}$, then $e^{2i\phi_1}$ is fixed by $B$. However, we also have that if $\arg \widetilde{H}(e^{i\phi_1}) = \phi_1 + \pi$, then $e^{2i\phi_1}$ is fixed by $B$. On the other hand, if $e^{i\phi_0}$ is fixed by $B$ and $\phi_1 \in \{ \phi_0/2, \phi_0/2 + \pi \}$ then we again conclude that $\arg \widetilde{H}( e^{i\phi_1}) \in \{ \phi_1, \phi_1 +\pi \}$. However, since $n$ is odd, there are two possibilities: either $e^{i\phi_1}$ and $-e^{i\phi_1}$ are both fixed points of $\widetilde{H}$ or they are switched by $\widetilde{H}$. The claim then follows.
\end{proof}

\begin{lemma}
\label{lem:Ij}
Let $g$ be an odd degree circle endomorphism with $g(e^{i\phi}) = -g(-e^{i\phi})$ and fixed points $w_1,w_1+\pi,w_2,w_2+\pi,\ldots,w_k,w_k+\pi$ with $0 \leq \arg w_1 \leq  \ldots \leq \arg w_k <\pi$. For $j=1,\ldots,k$ denote by $I_j$ the arc between $w_j$ and $w_{j+1}$, where we identify $w_{k+1}$ with $w_1+\pi$. Then $g$ either maps $I_j$ onto $I_j$ bijectively or $g$ maps $I_j$ onto $\partial \D$ by wrapping around exactly once. In the first case, this can only happen if one endpoint of $I_j$ is an attracting or neutral fixed point of $g$ and no point in $I_j$ can be mapped on its antipode. In the second case, there is at least one point in $I_j$ which maps onto its antipode.
\end{lemma}

\begin{proof}
With the hypotheses as above, either $g$ covers $\partial \D \setminus I_j$ with multiplicity $n_j \geq 1$ or $g(I_j) = I_j$ in which case we take $n_j =0$. With this set-up, we have
\begin{equation}
\label{eq:nj} 
n = \operatorname{deg}(g) = 2\sum_{j=1}^kn_j +1.
\end{equation}
Suppose that some $n_j \geq 2$. Then the pre-image $g^{-1}(I_j)$ consists of $n_j+1$ disjoint intervals, two of which have endpoints coinciding with the endpoints of $I_j$. Hence at least one of these pre-images, say $U$, does not have an endpoint coinciding with an endpoint of $I_j$. Then $g^{-1} : I_j \to U$ with $\overline{U}$ contained in the interior of $I_j$ and so there is a fixed point of $g$ contained in $U$. This contradicts the fact that there are no fixed points between $w_j$ and $w_{j+1}$ and so we conclude that $n_j$ can only be $0$ or $1$.

If $n_j=0$ then there are no points in $I_j$ which are switched under $\widetilde{H}$. If both endpoints of $I_j$ are repelling fixed points of $g$, then there exists an arc $V$ whose closure is contained in the interior of $I_j$ and so that $\overline{g(V)} \subset \operatorname{int} V$. Therefore $g$ must have another fixed point in the interior of $I_j$, which is a contradiction. Hence one of the endpoints is not repelling.

If $n_j=1$, then as above $U = g^{-1}(-I_j)$ is contained in the interior of $I_j$, where $-I_j = \{ -z :z\in I_j\}$. Therefore there exists at least one point $w$ in $I_j$ such that $g(w) = -w$, see Figure \ref{fig:2}. This completes the proof.

\begin{figure}[h]
\begin{center}
\includegraphics{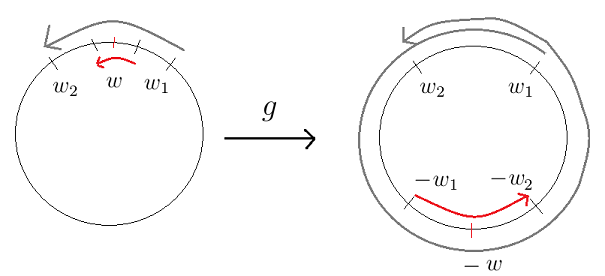}
\caption{The  case when $n_j=1$.}
\label{fig:2}
\end{center}
\end{figure}
\end{proof}

If $n$ is odd, then $\D \setminus \overline{\widetilde{\mathcal{E}_n}}$ consists of an even number of components which we denote by $C_j$ for $j=0,\ldots,n-2$ taken anticlockwise from the positive real axis. It follows from Theorem \ref{thm:F1} (vi)  that $C_j$ contains a ray from $\left ( \frac{n-1}{n+1} \right ) \exp [2i j\pi /(n-1)]$ to $\exp [2i j\pi /(n-1)]$, see Figure \ref{fig:6}.

\begin{figure}[h]
\begin{center}
\includegraphics{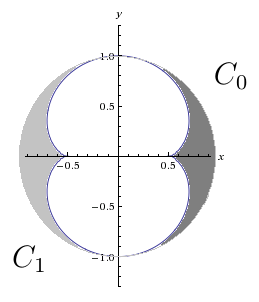}
\caption{The  case when $n=3$: $C_0$ is the darker region on the right and $C_1$ is the lighter region on the left.}
\label{fig:6}
\end{center}
\end{figure}

\begin{theorem}\label{thm:H2}
Let $n \geq 3$ be odd, $K\geq 1$, $\theta \in(-\pi/2,\pi/2]$ and let $\mu = e^{2i\theta} \left (\frac{K-1}{K+1} \right )$.  
\begin{enumerate}[(i)]
\item If $-\mu \in \widetilde{\mathcal{E}_n}$, then $H$ is elliptic and $H$ has $n-1$ fixed rays and $n-1$ switched rays.
\item If $-\mu$ is in the relative boundary of $\widetilde{\mathcal{E}_n}$ in $\D$, then $H$ is parabolic and there are two subcases:
\begin{enumerate}[(a)]
\item If $-\mu$ is on the boundary of $C_j$ for $j$ even, then the Denjoy-Wolff point of $B$ corresponds to a pair of rays fixed by $H$.
\item If $-\mu$ is on the boundary of $C_j$ for $j$ odd, then the Denjoy-Wolff point of $B$ corresponds to a pair of rays switched by $H$.
\end{enumerate}
\item If $-\mu \in \D \setminus \overline{\widetilde{\mathcal{E}_n}}$, then $H$ is hyperbolic and there are again two subcases:
\begin{enumerate}[(a)]
\item If $-\mu\in C_j$ for $j$ even, then the Denjoy-Wolff point of $B$ corresponds to a pair of rays fixed by $H$. There are $n+3$ fixed rays and $n-1$ switched rays of $H$.
\item If $-\mu\in C_j$ for $j$ odd, then the Denjoy-Wolff point of $B$ corresponds to a pair of rays switched by $H$. There are $n-1$ fixed rays and $n+3$ switched rays of $H$.
\end{enumerate}
\end{enumerate}
Further, if $n,\theta$ are fixed, there exists $K_{\theta}\in (1,\infty ]$ such that 
\begin{enumerate}[(i)]
\item for $1\leq K<K_{\theta}$, $H$ is elliptic and $H$ has $n-1$ fixed rays;
\item for $K=K_{\theta}$, $H$ is parabolic and $H$ has at most $n$ fixed rays;
\item for $K>K_{\theta}$, $H$ is hyperbolic and $H$ has either $n-1$ or $n+3$ fixed rays, depending on whether $-\mu \in C_j$ for $j$ odd or even.
\end{enumerate}
\end{theorem}

Note that in the parabolic case, $B$ may have less than $n$ fixed points, so $H$ will correspondingly have less fixed and switched rays. For example, $\left ( \frac{z-1/3}{1-z/3} \right )^2$ is parabolic and has only one fixed point on $\partial \D$ at $z=1$.

\begin{proof}[Proof of Theorem \ref{thm:H2}]
First, if $-\mu \in \widetilde{\mathcal{E}_n}$, then $B$ is elliptic and every fixed point on $\partial \D$ is repelling. By Lemma \ref{lem:Ij}, for every interval $I_j$ we are in the $n_j =1$ case. Hence there are $n-1$ fixed rays and $n-1$ switched rays of $H$.

Next, if $-\mu \in \D \setminus \overline{\widetilde{\mathcal{E}_n}}$, then $-\mu$ is in some component $C_j$. Since varying the parameters in $C_j$ moves fixed points of $B$ and fixed rays of $H$ continuously, it is enough to check what happens on the ray with argument $\exp [2i j\pi /(n-1)]$. 
To that end, let $K>n$, $\theta = j\pi/(n-1)$, recalling that the argument of $\mu$ is $2\theta$, and consider $B$ and $H$ with parameter $-\mu$. It is not hard to check that
\[ B(e^{2ij\pi /(n-1)}) = \left ( \frac{ e^{2ij\pi /(n-1)}(1+(K-1)/(K+1)) }{1+(K-1)/(K+1)} \right )^n = e^{2ijn\pi /(n-1)} = e^{2ij\pi /(n-1)},\]
and so $e^{2ij\pi /(n-1)}$ is fixed by $B$. Further,
\[B'(e^{2ij\pi /(n-1)}) =\frac{n}{K} <1,\]
and so this is the Denjoy-Wolff point of $B$. The next question is whether the corresponding rays of $H$, with argument $j\pi/(n-1)$ and $j\pi/(n-1)+\pi$ are fixed or switched by $H$. We have
\begin{align*} 
\arg H(e^{ij\pi /(n-1)}) &= \arg [ h_{K,\theta}(e^{ij\pi /(n-1)}) ]^n \\
&= \frac{ nj \pi}{n-1}\\
&= \frac{j\pi}{n-1} + j\pi.
\end{align*}
Therefore if $j$ is even, the rays are fixed and if $j$ is odd, the rays are switched.

Suppose $j$ is even and these rays are fixed. Then applying Lemma \ref{lem:Ij} to $\widetilde{H}$, we see that $\widetilde{H}$ has a pair of attracting fixed points and the rest are repelling and so there are two intervals where $n_j =0$, and for the others we must have $n_j =1$. Hence by Lemma \ref{lem:fixpts2} and using the fact $B$ has $n+1$ fixed points in this case, $\widetilde{H}$ has $n+3$ fixed points and $n-1$ switched points.

Similarly, if $j$ is odd, then $\widetilde{H}$ has $n-1$ fixed points and $n+3$ switched points.
The parabolic case is similar and so we omit the proof.
\end{proof}

\begin{example}
To illustrate the case when $n$ is odd, consider the example $H(z) = [h_{K,\theta}(z)]^3$, where first $\theta =0$ and $K>3$. Then there are $6$ fixed points of $\widetilde{H}$, including the two attracting fixed points $\pm 1$ arising from the Denjoy-Wolff point $z=1$ of $B$. The immediate attracting domains are bounded by the other pairs of fixed points. The points $\pm i$ are switched by $\widetilde{H}$. 

On the other hand, if $\theta = \pi/2$ and $K>3$, then now $\pm i$ are the points arising from the Denjoy-Wolff point $z=-1$ of $B$. They are still switched by $\widetilde{H}$ and now the immediate attracting region is bounded by pairs of points which are also switched. This means that $\pm 1$ are the only fixed points.
\end{example}

\section{Attracting and repelling fixed rays}

\subsection{Density of pre-images}

We may classify fixed rays of $H$ as attracting, repelling or neutral depending on whether the corresponding fixed point of $B$ is attracting, repelling or neutral. This classification also holds for opposite rays that are switched by $H$.

\begin{definition}
Suppose that $-\mu \notin \widetilde{\mathcal{M}_n}$. Then the corresponding Blaschke product $B$ from \eqref{eq:B} has a Denjoy-Wolff point $z_0 \in \partial \D$. If $n$ is even, define $R_0$ to be the corresponding fixed ray with argument $\phi_0$ and denote by $\Lambda$ the basin of attraction of $R_0$, that is,
\[ \Lambda = \{z\in \C : \arg (H^n(z)) \to \phi_0)\} .\] 
In this case, we will call $R_0$ the Denjoy-Wolff ray of $H$.

If $n$ is odd, then the Denjoy-Wolff point of $B$ corresponds to a pair of opposite rays $R_0,R_1$ with arguments $\phi_0$ and $\phi_0 + \pi$ which are either both fixed or both swapped by $H$. In this case, the basin of attraction $\Lambda$ is
\[ \Lambda = \{ z\in \C : \arg (H^{2n}(z)) \to \phi_0 \} \cup \{ z\in \C : \arg (H^{2n}(z)) \to \phi_0+\pi \}.\]
\end{definition}

The immediate basin of attraction $\Lambda_0$ is the component of $\Lambda$ that contains $R_0$ in even degree case or the two components of $\Lambda$ that contain $R_0$ and $R_1$ in the odd degree case.
Recall the connectedness locus $\mathcal{M}_n$ in parameter space of unicritical Blaschke products, and that the Blaschke product of form \eqref{eq:B} has parameter $-\mu$.

\begin{theorem}
\label{thm:basin}
Let $K>1$, $\theta \in (-\pi/2,\pi/2]$, $n\geq 2$ and $\mu = e^{2i\theta}\left ( \frac{K-1}{K+1} \right )$. If $H=H_{K,\theta,n}$
and $-\mu \in \mathcal{M}_n$, then for any ray $R_{\phi}$, $\{ H^{-k}(R_{\phi}) \} _{k=0}^{\infty}$ is dense in $\C$. On the other hand, if $-\mu \notin \mathcal{M}_n$, then $\Lambda$ is dense in $\C$, where $\Lambda$ is the basin of attraction defined above.
\end{theorem}

This theorem can be interpreted as saying either the backward orbit of a ray is dense, or the backward orbit of $\Lambda_0$ is dense depending on whether or not $-\mu \in \mathcal{M}_n$. To prove this we first need a result on circle endomorphisms.

\subsection{Relating the dynamics of a circle endomorphism $g$ to that of $T(g)$}

We need to study how the dynamics of the Blaschke product $B$ on $\partial \D$ and the dynamics of $\widetilde{H}$ are related. By Theorem \ref{thm:agree}, $\widetilde{H} = T(B)$ and so if $S(z) = z^2$ then we have the functional equation
$S\circ B = \widetilde{H} \circ S$.

\begin{definition}
Let $g:\partial \D \to \partial \D$ be a degree $m$ endomorphism. For such a map, define 
$J(g)$ to be the set of $z\in \partial \D$ such that for all neighbourhoods $U$ of $z$, there exists $N\in \N$ such that $g^N(U) = \partial \D$. Further, define $F(g)$ to be the complement of $J(g)$ in $\partial \D$, that is, the set of $z\in \partial \D$ such that there exists a neighbourhood $U$ of $z$ such that for all $N\in \N$, $g^N(U)$ omits an exceptional set $E$ containing at least one point. 
\end{definition}

Clearly, the exceptional set $E$ contains $J(g)$, as long as $J(g)$ is non-empty. 
If $g$ is the restriction of a finite Blaschke product to $\partial \D$, then $J(g)$ and $F(g)$ are the Julia set and the Fatou set restricted to $\partial \D$ respectively.

\begin{lemma}
\label{lem:rescale}
We have $J(T(g)) =  S^{-1}(J(g))$.
\end{lemma}

\begin{proof}
First suppose that $x\in J(g)$. Then given any neighbourhood $U$ of $x$, there exists $N\in \N$ such that $g^N(U) = \partial \D$.
Let $y\in S^{-1}(x)$ and find a neighbourhood $V$ of $y$ so that $V$ contains the component of $S^{-1}(U)$ containing $x$. Then $g^N(S(V)) \supset g^N(U) =\partial \D$. Using the functional equation, this means that $S([T(g)]^N(V)) = \partial \D$. Since $[T(g)]^N(V)$ is an arc, this means that $[T(g)]^N(V)$ contains an arc of length $\pi$. Hence $[T(g)]^{N+1}(V) = \partial \D$ and so $y \in J(T(g))$.

On the other hand, suppose that $x\in F(g)$. Then there exists a neighbourhood $U$ of $x$ so that for every $N\in \N$, $g^N(U)$ omits an exceptional set $E$. Let $y \in S^{-1}(x)$ and find a neighbourhood $V$ of $y$ so that $V$ is contained in the component of $S^{-1}(U)$ containing $x$. Then $g^N(S(V)) \subset f^N(U) \subset \partial \D \setminus E$. Again using the functional equation, this means that $S([T(g)]^N(V)) \subset \partial \D \setminus E$ and hence $[T(g)]^N(V) \subset S^{-1}( \partial \D\setminus E)$, which contains at least two points. Since this is true for every $N$, $y\in F(T(g))$.
\end{proof}

Denote by $O_g^-(z)$ the backward orbit of $z$ with respect to $g$.

\begin{lemma}
\label{lem:bwards}
If $z_0 \in \partial \D$, then $J(g) \subset \overline{O_g^-(z_0)}$.
\end{lemma}

\begin{proof}
Let $z_0 \in \partial \D$, $z_1 \in J(g)$ and $U$ be any neighbourhood of $z_1$. Then there exists $N\in \N$ such that $g^N(U) = \partial \D$. In particular, it follows that there exists $z_2 \in U$ with $g^N(z_2) = z_0$. This proves the lemma.
\end{proof}

We can now prove Theorem \ref{thm:basin}.

\begin{proof}[Proof of Theorem \ref{thm:basin}]
First suppose that $-\mu \in \mathcal{M}_n$. Then by definition $J(B) = \partial \D$. Since $\widetilde{H} = T(B)$, then Lemma \ref{lem:rescale} implies that $J(\widetilde{H}) = \partial \D$. Further, Lemma \ref{lem:bwards} implies that if $z_0 \in \partial \D$ then $O_{\widetilde{H}}^-(z_0)$ is dense in $\partial \D$. Interpreting this in terms of $H$, the backward orbit of any ray $R_{\phi}$ under $H$ is dense in $\C$.

Next, if $-\mu \notin \mathcal{M}_n$, then $J(B)$ is a Cantor subset of $\partial \D$ and there is a Denjoy-Wolff point $z_0 \in \partial \D$ so that if $z \in F(B)$ then $B^m(z) \to z_0$. 
Lemma \ref{lem:rescale} implies that $J(\widetilde{H})$ is also a Cantor subset of $\partial \D$ and so $F(\widetilde{H})$ is dense in $\partial \D$. The Denjoy-Wolff point $z_0$ of $B$ corresponds to either a single fixed ray or a pair of rays that are either fixed or switched by $\widetilde{H}$, as discussed above. Interpreting this in terms of $H$, the basin of attraction $\Lambda$ is dense in $\C$.
\end{proof}

\section{Decomposition of the plane}

\subsection{Attracting basins and their boundary}

The dynamics of $H$ break up the plane into three dynamically interesting sets.

\begin{theorem}
\label{thm:sets}
Let $H$ be as in \eqref{eq:H}. Then the attracting basin of $0$, $\mathcal{A}(0)$, is star-like about $0$ and we may write
\[ \C = \mathcal{A}(0) \cup \partial I(H) \cup I(H),\]
where $I(H)$ denotes the escaping set.
In other words, the attracting basins of $0$ and $\infty$ respectively form two completely invariant domains with boundary $\partial I(H)$ a Jordan curve.
\end{theorem}

\begin{proof}[Proof of Theorem \ref{thm:sets}]

Fix $K>1$, $\theta \in (-\pi/2, \pi/2]$, $n\geq 2$ and let $H=H_{K,\theta,n}$ be defined by \eqref{eq:H}. By \cite[Theorem 4.3]{FG}, since $H$ is a composition of a bi-Lipschitz map and a polynomial, the escaping set $I(H)$ is a connected, completely invariant, open neighbourhood of infinity and $\partial I(H)$ is a completely invariant closed set. It is clear that $0$ is a topologically attracting fixed point of $H$ and so the basin of attraction $\mathcal{A}(0)$ is completely invariant and open.

Let $R_{\phi}$ be a fixed ray of $H$. Then on $R_{\phi}$, we have
\[ H(re^{i\phi}) = \alpha r^n e^{i\phi},\]
where $\alpha =( 1+ (K^2-1)\cos^2(\phi - \theta) )^{n/2}$ by the polar form \eqref{eq:polarH} of $H$. For $r=r_{\phi}:=\alpha^{\frac{1}{1-n}}$, this point is fixed, for $r>r_{\phi}$ the point is in $I(H)$ and for $r<r_{\phi}$, the point is in $\mathcal{A}(0)$. By complete invariance, any pre-image of $R_{\phi}$ breaks up into $\mathcal{A}(0), I(H)$ and $\partial I(H)$ in the same way.

Suppose that $-\mu \in \mathcal{M}_n$ so that $J(B) = \partial \D$. Then by Theorem \ref{thm:basin}, if $R_{\phi}$ is any fixed ray of $H$, its pre-images under $H$ are dense in $\C$. Since $\mathcal{A}(0)$ and $I(H)$ are open, this proves the result in this case.

Next, if $-\mu \notin \mathcal{M}_n$, then $J(B)$ is a Cantor subset of $\partial \D$ and $F(B) \cap \partial \D$ is dense in $\partial \D$. By Theorem \ref{thm:basin}, the basin of attraction $\Lambda$ is dense in $\C$. Suppose first that $n$ is even and that $R_{\varphi} \in \Lambda$. Then $H^m(R_{\varphi}) \to R_{\phi}$ where $R_{\phi}$ is the attracting fixed ray. Since $\mathcal{A}(0)$ and $I(H)$ are open, it is not hard to see that $R_{\varphi}$ decomposes in the same way that $R_{\phi}$ does. Since $\Lambda$ is dense in $\C$, the openness of $\mathcal{A}(0)$ and $I(H)$ again imply the result in this case.

The case where $n$ is odd and the Denjoy-Wolff point of $B$ corresponds to a pair of rays $R_{\phi},R_{\phi + \pi}$ which are either fixed or switched follows similarly by considering $H^{2m}(R_{\varphi})$. This sequence converges to $R_{\phi}$ or $R_{\phi + \pi}$ and we then proceed as above.
\end{proof}

It would be interesting to know the regularity of $\partial I(H)$. For example, is it a quasi-circle? See the computer pictures generated by Doug Macclure in Figures \ref{fig:3} and \ref{fig:4} for examples.

\begin{figure}[h]
\begin{center}
\includegraphics[width = 5in]{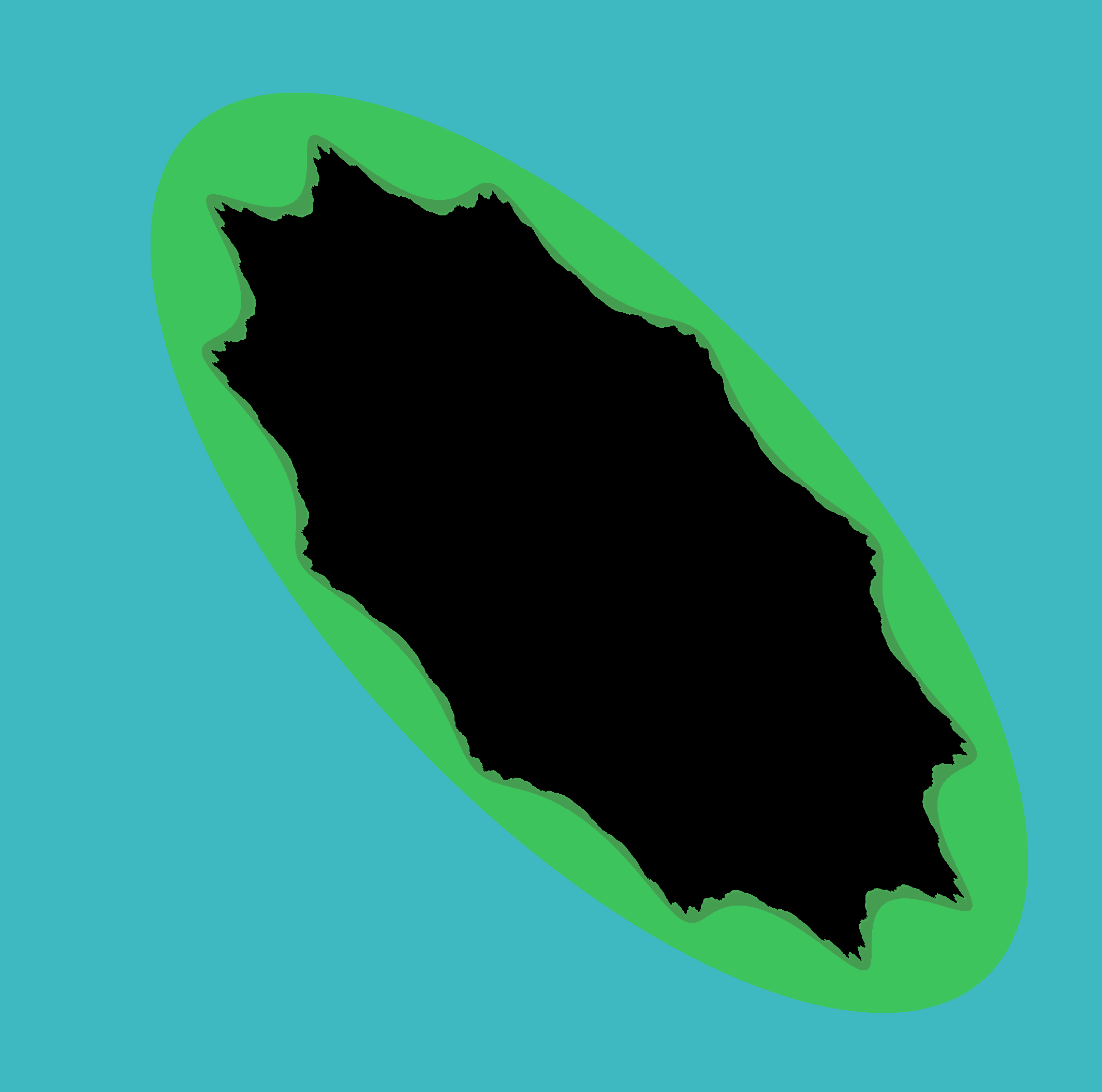}
\caption{The dynamics of $H$ with $K=2.25$, $\theta = 0.75$ and $n=6$.}
\label{fig:3}
\end{center}
\end{figure}

\begin{figure}[h]
\begin{center}
\includegraphics[width = 5in]{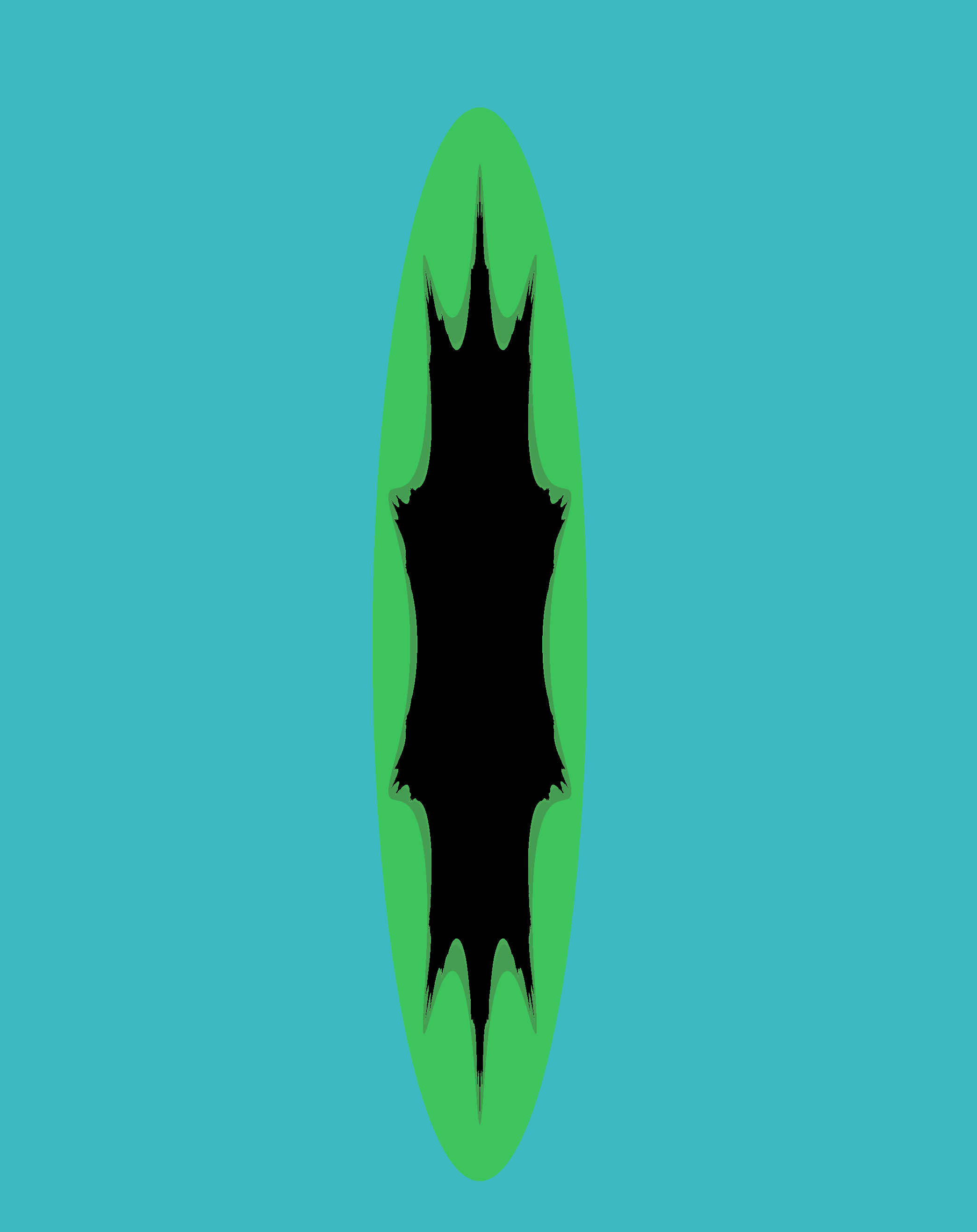}
\caption{The dynamics of $H$ with $K=5$, $\theta = 0$ and $n=5$.}
\label{fig:4}
\end{center}
\end{figure}

\subsection{Julia sets}

In \cite{B2}, the Julia set for quasiregular mappings of polynomial type is defined, but only when the degree is larger than the distortion. For such mappings, $J(f)$ is defined by
\[ J(f) = \{ x\in \R^n : \R^n \setminus O^+(U) \text{ has capacity zero, for all neighbourhoods } U \text{ of } x \},\]
where $O^+(U)$ denotes the forward orbit of the set $U$. We omit the definition of capacity zero here, but remark that such sets must be necessarily of Hausdorff dimension zero \cite[Corollary VII.1.16]{Rickman}. We next show that all such mappings $H$ to which this definition of the Julia set applies are elliptic.

\begin{corollary}
\label{cor:F}
Let $n\geq 2$ and $\theta \in(-\pi/2,\pi/2]$. Then $K_{\theta}$ defined in Theorems \ref{thm:H1} and \ref{thm:H2} satisfies $K_{\theta}\geq n$. 
\end{corollary}

\begin{proof}
This follows immediately from Theorem \ref{thm:F1} (iii), that $\widetilde{\mathcal{E}_n}$ contains an open disk of radius $\frac{n-1}{n+1}$, and the fact that $| \mu | = \frac{K-1}{K+1}$.
\end{proof}

This shows that the Julia set definition can only apply when $H$ is elliptic. It follows fairly straightforwardly that in fact if $-\mu \in \widetilde{\mathcal{M}_n}$ then all points on $\partial I(H)$ have the blowing-up property of the Julia set and by Theorem \ref{thm:sets} these are the only such points.

On the other hand, if $-\mu \notin \widetilde{\mathcal{M}_n}$, then there are points on $\partial I(H)$ without the blowing-up property in the Julia set definition. In fact, the only points with the blowing-up property are those on $\partial I(H)$ that arise from the Julia set of $B$, which we recall is a Cantor set. This gives another class of examples where the boundary of the escaping set and the set with a blowing-up property do not agree, c.f. \cite{B1}.

\section{Unbounded distortion of the iterates}

\subsection{Nowhere uniform quasiregularity}

The dynamics of mappings of the form $H$ are only of independent interest if the distortion of the iterates is unbounded. This is because every uniformly quasiregular mapping of the plane is a quasiconformal conjugate of a holomorphic mapping. This means that the iteration of uniformly quasiregular mappings of the plane yields nothing new compared to complex dynamics. We will next show that mappings of the form $H$ satisfy a condition that is slightly stronger than not being uniformly quasiregular. We recall the following definition from \cite{FF}.

\begin{definition}
\label{nuqr}
Given a plane domain $U$, a quasiregular mapping $f:U \to \C$ is called \emph{nowhere uniformly quasiregular} if for every
$z\in U$, we have $K_z(f^m)$ is unbounded as $m\to \infty$, where
\[ K_z(f) = \inf \max \{ K_f(w):w \in U \},\]
where $K_{f}(w)$ denotes the distortion of $f$ at $w$ and the infimum is taken over all neighbourhoods $U$ of $z$.
\end{definition}

For example, the quasiconformal mapping $f(x+iy) = Kx+iy$ is easily seen to be nowhere uniformly quasiregular for any $K>1$.

\begin{theorem}\label{s2t3}
Let $n\geq 2$, $\theta \in(-\pi/2,\pi/2]$ and $K>1$. Then $H(z) = (h_{K,\theta}(z))^n$ is nowhere uniformly quasiregular. 
\end{theorem}

Before proving this, we need to recall some material on M\"obius transformations.

\subsection{M\"obius transformations}

Every M\"obius transformation of the unit disk can be written in the form
\[ A(z) = \frac{az+b}{cz+d}, \quad ad-bc =1, \:\:\: a+d\in \R.\]
The mapping can be represented by the matrix $\begin{pmatrix} a & b \\ c & d \end{pmatrix} $ which has trace-squared $\tau(A) = (a+d)^2$. The value of $\tau$ classifies the dynamical behaviour of $A$:
\begin{enumerate}[(i)]
\item if $\tau(A) \in[0,4)$ then $A$ is {\it elliptic} and there exists a fixed point $z_0 \in \D$;
\item if $\tau(A) = 4$, then $A$ is {\it parabolic} and there exists one fixed point $z_0 \in \partial \D$;
\item if $\tau(A) >4$, then $A$ is {\it hyperbolic} and there exist two fixed points in $\partial \D$.
\end{enumerate}

We will need the following theorem on the composition of varying M\"obius maps.

\begin{theorem}\label{sMlMM}\cite{MM}
Let $A,A_j$ be hyperbolic M\"{o}bius maps of $\D$ such that $A^m(z)\to w_0 \in\partial\D$ as $m\to\infty$ for all $z\in \D$ and $A_j \to A$ locally uniformly as $j\to\infty$.
Suppose we have a sequence $t_m$ of hyperbolic M\"{o}bius maps of $\D$ defined by
\begin{equation*}
t_m(z)=A_1\circ A_2\circ\ldots\circ A_m(z),
\end{equation*}
Then $t_m(z) \to w_0$ as $n\to \infty$ for all $z\in\D$.
\end{theorem}

\subsection{On fixed rays and switched rays}

We now show that if $z$ is on a fixed ray of $H$ or a pair of switched rays, then the distortion of the iterates of $H$ is unbounded.

Recall that the complex dilatation of a quasicregular mapping is given by
\[ \mu_f(z) = \frac{f_{\overline{z}}(z)}{f_z(z)}.\]
The composition formula for complex dilatations is (see for example \cite{FM}):
\begin{equation} 
\label{eq:comp}
\mu_{g\circ f}(z) = \frac{ \mu_f(z) + r_f(z)\mu_g(f(z))}{1+r_f(z)\overline{ \mu_f(z)}\mu_g(f(z))},
\end{equation}
where $r_f(z) = \overline{f_z(z)}/f_z(z)$.
Hence $H(z)=[h(z)]^n$, we see that $\mu_H$ is the constant
\[ \mu_H(z) \equiv e^{2i\theta}\left ( \frac{K-1}{K+1} \right ) =: \mu.\]

\begin{lemma}
\label{lem:6.1}
For $m\geq 1$,
\[ \mu_{H^m}(z) = \frac{ \mu_H + e^{-2(n-1)i\arg h(z)}\mu_{H^{m-1}}(H(z))}{1+ e^{-2(n-1)i\arg h(z)} \overline{\mu_H}\mu_{H^{m-1}}(H(z))}.\]
Next, if $z$ is on a fixed ray $R_{\phi}$ of $H$, then $\mu_{H^m}(z) = A^m(\mu)$, where $A$ is the M\"obius transformation
\[ A(w)  = \frac{ e^{-2(n-1)i(\phi +2k\pi) /n }w + \mu }{1+  e^{-2(n-1)i(\phi +2k\pi) /n }\overline{\mu} w},\]
for some $k\in \{0,1,\ldots,n-1\}$.
Finally, if $z$ is on a pair of rays $R_{\phi},R_{\phi+\pi}$ switched by $H$, then $\mu_{H^m}(z) = A^m(\mu)$, where $A$ is the M\"obius transformation
\[ A(w)  = \frac{ e^{-2(n-1)i(\phi +(2k+1)\pi) /n }w + \mu }{1+  e^{-2(n-1)i(\phi +(2k+1)\pi) /n }\overline{\mu} w},\]
for some $k\in \{0,1,\ldots,n-1\}$.
\end{lemma}

\begin{proof}
For the first part, we just need to calculate $r_H(z)$ and apply \eqref{eq:comp} to $H^m = H^{m-1}\circ H$. We can calculate that
\begin{equation}
\label{eq:Hz}
H_z(z) = nh(z)^{n-1}h_z(z) = \frac{n(K+1)h(z)^{n-1}}{2},
\end{equation}
from which it follows that $r_H(z) = e^{-2(n-1)i\arg h(z)}$. 

Next, suppose that $R_{\phi}$ is fixed by $H$ and $z\in R_{\phi}$. Since $\arg H(z) = \phi$, we must have $\arg h(z) = (\phi +2k\pi)/n$ for some $k\in \{0,1,\ldots,n-1\}$. Hence $r_H$ is constant on $R_{\phi}$ and it follows by induction that $\mu_{H^m}$ must be constant on $R_{\phi}$ and given by the desired iterated M\"obius map evaluated at $w=\mu$.

For the final case, if $R_{\phi}$ and $R_{\phi+\pi}$ are switched by $H$ and say $z\in R_{\phi}$, then $h(z) \in R_{\phi +\pi}$ and so $\arg h(z) = (\phi +2k\pi + \pi)/n$ for some $k\in \{0,1,\ldots,n-1\}$. Since $h$ maps pairs of opposite rays onto pairs of opposite rays, if $z\in R_{\phi + \pi}$, then $\arg h(z) = \pi +  (\phi +2k\pi + \pi)/n$. Hence for $z$ on either of the swapped rays, we have
$r_H(z) = e^{-2(n-1)i(\phi +(2k+1)\pi)}$, and the claim follows. 
\end{proof}

Suppose that $R_{\phi}$ is a fixed ray of $H$. Then there is an associated M\"obius map given by $A$ in Lemma \ref{lem:6.1} which we can write as
\[ A(z) = \frac{\alpha z + \mu}{1+\alpha \overline{\mu}z} = \frac{ \alpha z/ D + \mu /D }{1/D + \alpha \overline{\mu} z / D},\]
where $\alpha = e^{-2(n-1)i(\phi +2k\pi) /n }$ and $D = e^{-(n-1)i(\phi +2k\pi) /n }(1-|\mu |^2)^{1/2}$. The point is that the matrix representing this latter way of writing $A$ has determinant $1$ and has trace-squared equal to
\begin{equation}\label{eq:tau} \tau(A) = \frac{ (\alpha+1)^2}{D^2}.\end{equation}
Since M\"obius transformations and their dynamical behaviour are classified by $\tau$, we can use $A$ to get information about how $H^m$ behaves on $R_{\phi}$.

\begin{lemma}
\label{lem:nuqrray}
Let $R_{\phi}$ be a fixed ray of $H$. Then, with $k$ as above,
\[ \tau(A)= \frac{(K+1)^2\cos^2((n-1)(\phi + 2k\pi)/n)}{K} \geq 4.\]
Hence $A$ is parabolic or hyperbolic and consequently $|A^m(z)| \to 1$ for any $z\in \D$.
\end{lemma}

\begin{proof}
Using \eqref{eq:tau},
\begin{align*} 
\tau(A) &=  \frac{ (\alpha+1)^2}{D^2} \\
& = \frac{ (e^{-2(n-1)i(\phi +2k\pi) /n } + 1)^2 }{e^{-2(n-1)i(\phi +2k\pi) /n }(1-|\mu|^2)} \\
&= \frac{ e^{2(n-1)i(\phi +2k\pi) /n } + 2 + e^{-2(n-1)i(\phi +2k\pi) /n } }{1- \left ( \frac{K-1}{K+1} \right )^2 } \\
& = \frac{(K+1)^2 (e^{(n-1)i(\phi +2k\pi) /n } + e^{-(n-1)i(\phi +2k\pi) /n } )^2 }{4K} \\
&= \frac{(K+1)^2\cos^2((n-1)(\phi + 2k\pi)/n)}{K},
\end{align*}
as claimed. We have to show that this expression is always at least $4$. To that end, if $R_{\phi}$ is a fixed ray, then
\begin{equation}\label{eq:6.2a} 
\tan \left ( \frac{\phi + 2k\pi}{n} - \theta \right ) = \frac{\tan(\phi - \theta)}{K}.
\end{equation}
Next, \eqref{eq:6.2a} and the tangent addition formula give
\begin{align*}
\tan \left ( \frac{ (n-1)(\phi+2k\pi)}{n} \right ) &= \tan \left ( (\phi+2k\pi - \theta) - ((\phi + 2k\pi )/n - \theta ) \right ) \\
&= \frac{ \tan(\phi +2k\pi - \theta) - \tan((\phi + 2k\pi )/n - \theta ) }{1+ \tan(\phi +2k\pi - \theta)\tan((\phi + 2k\pi )/n - \theta )} \\
&= \frac{ (K-1) \tan(\phi + 2k\pi -\theta )}{K+\tan^2(\phi +2k\pi - \theta )}.
\end{align*}
Consider the function
\[ F(T) = \frac{ (K-1)T}{K+T^2}.\]
An elementary calculation shows that $|F(T)| \leq \frac{K-1}{2\sqrt{K}}$. Then
since $\cos^2(x) = (1+\tan^2(x))^{-1}$, we get
\begin{align*}
\tau(A) &=  \frac{(K+1)^2\cos^2((n-1)(\phi + 2k\pi)/n)}{K} \\
&= \frac{ (K+1)^2 }{K\left (1 + \tan^2\left ( \frac{ (n-1)(\phi+2k\pi)}{n} \right ) \right )}\\
&\geq \frac{ (K+1)^2}{K(1+ (K-1)^2/4K )} \\
&= \frac{4(K+1)^2}{4K + (K-1)^2} = 4.
\end{align*}
Hence by the classification of M\"obius transformations, $A$ is parabolic or hyperbolic. This means that there exists $\nu \in \partial \D$ such that for every $z\in \D$, $A^m(z) \to \nu$.
\end{proof}

This shows that on fixed rays, $H$ is nowhere uniformly quasiregular. The case for switched rays follows analogously with $2k$ replaced with $2k+1$, recalling Lemma \ref{lem:6.1}.

\subsection{Everywhere else}

 To show that $H$ is nowhere uniformly quasiregular everywhere, we combine Lemma \ref{lem:nuqrray} with Theorem \ref{thm:basin} on the density of either the pre-images of a fixed ray if $-\mu \in \mathcal{M}_n$, or the density of the basin of attraction $\Lambda$ otherwise.

We deal first with the case that $-\mu \in \mathcal{M}_n$.

\begin{lemma}
\label{lem:nuqr1}
Suppose that $-\mu \in \mathcal{M}_n$. Then $H$ is nowhere uniformly quasiregular.
\end{lemma}

\begin{proof}
Let $\phi$ be a fixed ray of $H$ and fix $z\in\C$. Then $\mathcal{P} = \{ H^{-k}(R_{\phi}) \}$ is dense in $\C$ by Theorem \ref{thm:basin}. If $z$ lies on a ray $R_\varphi\in\mathcal{P}$ then there exists $m$ such that $H^m(R_\varphi)=R_\phi$. That is, $H^k(z)$ lies on the ray $R_\phi$ for $k\geq m$. We can apply \eqref{eq:comp} to obtain:
\begin{equation}\label{s6e1}
\mu_{H^j\circ H^m}(z)=\frac{\mu_{H^m}(z) + r_{H^m}(z) \mu_{H^j}(H^m(z))}{1+r_{H^m}(z)\overline{\mu_{H^m}(z)}\mu_{H^j}(H^m(z))},
\end{equation}
for $j\geq 0$,
where $r_{H^m}(z)=\overline{(H^m)_{z}(z)}/(H^m)_z(z)$. Notice that $|r_{H^m}(z)|=1$ and that if we define
\begin{equation}\label{s6e2}
M(w):=r_{H^m}(z)\left(\frac{w + \mu_{H^m}(z)\overline{r_{H^m}(z)}}{1+\overline{[\overline{r_{H^m}(z)}\mu_{H^m}(z)]}w}\right),
\end{equation}
then $M$ is a M\"{o}bius map of the disk and further we see that
\[M[\mu_{H^j}(H^m(z))]=\mu_{H^j\circ H^m}(z),\]
for $j \geq 0$.
Using the fact that $H^{j+m}(z) \in R_{\phi}$ for $j \geq 0$, \eqref{s6e2} and Lemma \ref{lem:nuqrray} we see that \eqref{s6e1} becomes
\begin{equation}\label{s6e3}
\mu_{H^j\circ H^m}(z)=M(A^{j-1}(\mu_{H^m}(z))),
\end{equation}
for $j \geq 0$.
We know that $|A^j(w)|\to 1$ as $j\to\infty$ for any $w \in \D$ and that $|M(w)| \to 1$ as $|w|\to 1$, and so we have
\[|\mu_{H^k(z)}(z)|\to 1 \mbox{ as } k\to\infty.\]
Any neighbourhood $U\ni z$ trivially contains $z$ and so $K_z(H^k)$ is unbounded as $k \to \infty$ for any $z$ on a ray in $\mathcal{P}$.

Next suppose $z$ lies on a ray not in $\mathcal{P}$. As $\mathcal{P}$ is dense in $\C$, any neighbourhood $U\ni z$ must intersect a ray $R_\varphi\in\mathcal{P}$. Picking one such ray there must exist $m$ (depending on the neighbourhood $U$) such that $H^m(R_\varphi)=R_\phi$ and we can apply the same argument above to conclude $K_z(H^k)$ is unbounded as $k \to \infty$ for any $z\in\C.$
\end{proof}

We next turn to the case where $-\mu \notin \mathcal{M}_n$ for even degree.

\begin{lemma}
\label{lem:nuqr2}
Let $n$ be even and suppose that $-\mu \notin \mathcal{M}_n$. Then $H$ is nowhere uniformly quasiregular.
\end{lemma}

\begin{proof}
By hypothesis, the Blaschke product $B$ has a Denjoy-Wolff point $z_0$ on $\partial \D$ and so $H$ has a Denjoy-Wolff fixed ray $R_0$ with argument $\phi_0$ and corresponding basin of attraction $\Lambda$.

Fix $z\in\C$ and suppose that $z\in\Lambda$. Then the argument of $H^m(z)$ tends to the argument of the Denjoy-Wolff ray $R_0$ as $m\to\infty$.

We define the sequence $e^{i\phi_m} \in \partial \D$ by $H^m(z)\in R_{\phi_m}$. Then $\phi_m \to \phi_0$ as $m\to\infty$.
Again we use \eqref{eq:comp} to see that
\[\mu_{H^m}(z)=\mu_{H^{m-1}\circ H}(z) =\frac{\mu_{H}(z)+r_{H}(z)\mu_{H^{m-1}}(H(z))}{1+r_{H}(z)\overline{\mu_{H}(z)}\mu_{H^{m-1}}(H(z))}.\]
Recalling that $\mu_H$ is constant, we can write
\[ \mu_{H^m}(z) = A_1(\mu_{H^{m-1}}(H(z))),\]
where $A_1$ is the M\"{o}bius map
\[ A_1(w) = \frac{ \mu_H + r_H(z)w}{1+r_H(z)\overline{\mu_H}w}.\]
Using the same method, we may write
\[\mu_{H^{m-1}}(H(z)) = A_2(\mu_{H^{m-2}}(H^2(z))),\]
where $A_2$ is the M\"{o}bius map
\[ A_2(w) = \frac{ \mu_H + r_H(H(z))w}{1+r_H(H(z))\overline{\mu_H}w}.\]
By induction, we may write
\[ \mu_{H^m}(z) = A_1\circ A_2\circ\ldots\circ A_{m-1}(\mu_H(H^{m-1}(z))),\]
where each $A_i$ is the M\"{o}bius map given by
\[ A_i(w)=\frac{ \mu_H + r_H(H^{i-1}(z))w}{1+r_H(H^{i-1}(z))\overline{\mu_H}w}.\]
By \eqref{eq:Hz}, we have $H_z(z) = n(K+1)h(z)^{n-1}/2$, and so
\[ r_H(H^{j-1}(z)) = \exp ( -2i(n-1) \arg[h(H^{j-1}(z))] ).\]
As $j \to \infty$, we have $\arg [h(H^{j-1}(z))] \to \arg [h(re^{i\phi_0})]$ for any $r>0$ since $z\in \Lambda$.
Recalling Lemma \ref{lem:6.1}, there exists $k\in\{ 0,1,\ldots,n-1\}$ such that
\[ \arg [h(re^{i\phi_0})] = \frac{\phi_0 + 2k\pi}{n}.\]
Therefore as $j\to \infty$,
\[ r_H(H^{j-1}(z)) \to \exp \left ( \frac{-2(n-1)i(\phi_0+2k\pi)}{n} \right ),\]
and so $A_j$ converges to the M\"obius transformation $A$ given in Lemma \ref{lem:6.1}. By Lemma \ref{lem:6.1} and Lemma \ref{lem:nuqrray}, $A$ is either a parabolic or hyperbolic M\"obius transformation. Either way, there exists $w_0\in \partial \D$ such that $A^m(z) \to w_0$ for all $z\in \D$.

We can write
\begin{equation}\label{s6e6}
\mu_{H^m}(z)=A_1\circ A_2\circ\ldots\circ A_{m-1}(\mu_H)=:t_{m-1}(\mu_H).
\end{equation}
Applying Theorem~\ref{sMlMM} to the sequence $t_m$ given in \eqref{s6e6} we obtain that $|\mu_{H^m}(z)| \to 1$. This proves the lemma.
\end{proof}

We finally deal with the case that $-\mu \notin \mathcal{M}_n$ and $n$ is odd.

\begin{lemma}
\label{lem:nuqr3}
Let $n$ be odd and suppose that $-\mu \notin \mathcal{M}_n$. Then $H$ is nowhere uniformly quasiregular.
\end{lemma}

\begin{proof}
In this case, there are two opposite rays $R_0,R_1$ arising from the Denjoy-Wolff point $z_0$ of $B$ on $\partial \D$ that are either both fixed or both switched by $H$. There is an associated basin of attraction $\Lambda$ that is dense in $\C$. The proof is similar to Lemma \ref{lem:nuqr2} and so we omit the details. The only modification needed is to take into account the fixing or switching of $H$ on $\Lambda$. Since $r_H(z) = e^{-2(n-1)i\arg h(z)}$ it follows that $r_H(z) = r_H(-z)$ and so the proof for both of the cases is the same.

\end{proof}

The previous lemmas complete the proof of Theorem \ref{s2t3}.

\section{Local behaviour near fixed points}

\subsection{B\"ottcher coordinates}

With the results of the previous sections in hand, we can apply them to quasiregular mappings for which the complex dilatation is constant in the neighbourhood of a fixed point. To do this, we need to make use of a B\"ottcher coordinate for such a situation. Such coordinates were constructed when the fixed point has local index $2$ in \cite{FF1}. The method employed there works for any local degree.

\begin{theorem}
\label{thm:bottcher}
Let $U\subset \C$ be a domain, $f:U \to \C$ be quasiregular and $z_0 \in U$ be a fixed point of $f$ with local index $i(z_0,f) = n \geq 2$. Further suppose that there is a neighbourhood $U_1$ of $z_0$ on which $f$ has constant complex dilatation. Then there exists a domain $V \subset U_1$, $K\geq 1$, $\theta \in (-\pi/2, \pi/2]$ and a quasiconformal mapping $\psi :V \to \C$ such that
\[ \psi(f(z)) = H(\psi(z)) \quad z\in V,\]
where $H$ is given by \eqref{eq:H} with $K,\theta,n$ as above. Moreover, $\psi$ is asymptotically conformal as $z\to z_0$.
\end{theorem}

The details are rather involved, but follow the same theme (with minor changes) as the proof of the degree $2$ case in \cite{FF1}. For the convenience of the reader, we will sketch a proof.

\begin{proof}[Sketch proof of Theorem \ref{thm:bottcher}]

Let $f$ satisfy the hypotheses of the theorem. Then there exists a Stoilow decomposition of $f$ as $f=f_1\circ f_2$, where both $f_1,f_2$ fix $z_0$, $f_2$ is quasiconformal with constant complex dilatation $\mu$ in a neighbourhood of $z_0$ and we can arrange it so that $f_1$ is holomorphic with Taylor series
$f_1(z) = z_0 + (z-z_0)^n + \ldots$ where $n\geq 1$.

We can use logarithmic coordinates in a neighbourhood of $z_0$. Namely, given a function $g$ which fixes $z_0$ and is suitably well-behaved in a neighbourhood of $z_0$, we define its logarithmic transform (see for example \cite[p. 91]{Milnor}) by
\[ \widetilde{g}(w) = \log [g(z_0 + e^w)-z_0],\]
for $\Re(w) < \sigma$. The logarithmic trasform is only defined up to integer multiples of $2\pi i$, and satisfies $\widetilde{g_1\circ g_2} = \widetilde{g_1} \circ \widetilde{g_2}$.

If $P(z) = z^n$, then $\widetilde{P}(w) = nw$ and it is not hard to see that with $f_1$ as above,
\[\widetilde{f_1}(w) = nw + E(w), \]
where $|E(w)| = o(\Re(w))$ as $\Re(w) \to -\infty$.
By \cite[Lemma 3.12]{FF1}, 
\[ \widetilde{h_{K,\theta}}(w) = w +O(1),\]
where the bounded function depends only on the imaginary part of $w$. We also  $\widetilde{h_{K,\theta}^{-1}}(w) = w+O(1)$.

We now define a sequence $\phi_k$ of functions in $\Re(w) < \sigma$ as follows. Let
\[ \phi_1(z) =     \widetilde{h_{K,\theta}^{-1}} \left (  \frac{\widetilde{f_1}(\widetilde{f_2}(z))}{n}      \right ) .\]
This is the logarithmic transform of a suitably chosen branch of $H^{-1}\circ f$, where $H(z) = H_{K,\theta,n}(z-z_0)+z_0$. Undoing the logarithmic transform, we obtain a function $\psi_1$ defined in a neighbourhood of $z_0$ whose logarithmic transform is $\phi_1$.
For $k\geq 1$, define
\[ \phi_{k+1} (w) = \widetilde{h_{K,\theta}^{-1}} \left ( \frac{\phi_k( \widetilde{f_1}(\widetilde{f_2}(w) ) )  }{n} \right).\]
This is the logarithmic transform of a suitably chosen branch of $H^{-1} \circ \psi_k \circ f$, for some mapping $\psi_k$ whose logarithmic transform is $\phi_k$.

Arguing as in \cite{FF1}, it can be shown that
\[ \phi_k(w) = w + o(1), \text { as } \Re(w) \to -\infty,\]
in such a way that $\phi_k$ is asymptotically conformal as $\Re(w) \to -\infty$. It follows that $\psi_k$ is asymptotically conformal as $|z-z_0| \to 0$. Using the normality of uniformly bounded families of $K$-quasiconformal mappings, we obtain a quasiconformal limit $\psi$ of $\psi_k$ which is asymptotically conformal and conjugates $f$ to $H$. This mapping $\psi$ is the required B\"ottcher coordinate.
\end{proof}

\subsection{Fixed external rays}

We can use this B\"ottcher coordinate to describe the local dynamics near a fixed point where the complex dilatation is constant.

\begin{definition}
\label{def:rays}
With $f$ as in the hypotheses of Theorem \ref{thm:bottcher}, define the external ray $E_{\phi}$ of $f$ with angle $\phi \in [0,2\pi)$ as the image of the ray $R_{\phi} = \{ z : \arg(z-z_0) = \phi \}$  under the B\"ottcher coordinate $\psi$ from Theorem \ref{thm:bottcher}.
\end{definition}

The external ray is only initially defined in a neighbourhood of $z_0$ but can be continued to the immediate attracting basin of $z_0$.
We have described rays fixed by $H$ as repelling, attracting or neutral depending on whether the corresponding fixed points of the associated Blaschke product are repelling, attracting or neutral. Similarly, we may describe external rays fixed by $f$ as such. This allows us to describe curves fixed by $f$ which land at $z_0$.

\begin{corollary}
\label{cor:rays}
Let $f$ be as in the hypotheses of Theorem \ref{thm:bottcher} with constant complex dilatation $\mu = e^{2i\theta} \left ( \frac{K-1}{K+1} \right )$ and local index $n$. 
\begin{enumerate}[(a)]
\item If $n$ is even, then
\begin{enumerate}[(i)]
\item if $H$ is hyperbolic, $H$ has $n+1$ fixed external rays, one of which is attracting and the rest of which are repelling;
\item if $H$ is elliptic, then $H$ has $n-1$ fixed external rays, each of which are repelling;
\item if $H$ is parabolic, then $H$ has at most $n$ fixed external rays, one of which is neutral and the rest of which are repelling.
\end{enumerate}
\item If $n$ is odd, 
\begin{enumerate}[(i)]
\item if $H$ is hyperbolic, $H$ has $n+1$ pairs of fixed rays which are either fixed or switched;
\item if $H$ is elliptic, then $H$ has $n-1$ pairs of fixed rays which are either fixed or switched;
\item if $H$ is parabolic, then $H$ has at most $n$ pairs of fixed rays which are either fixed or switched.
\end{enumerate}
\end{enumerate}
If $-\mu \in \mathcal{M}_n$ then pre-images of any fixed external ray are dense in a neighbourhood of $z_0$. If $-\mu \notin \mathcal{M}_n$, then there is a basin of attraction corresponding to either one fixed external ray, a pair of fixed external rays or a pair of switched external rays. This basin of attraction is dense in the neighbourhood of $z_0$.
\end{corollary}

\begin{proof}
This follows by classifying the fixed and switched rays of $H(z) = [ h_{K,\theta}(z)]^n$ and then mapping them to the fixed external rays of $f$ by applying the appropriate B\"ottcher coordinate.
\end{proof}

The results of this paper may be strengthened if a more general B\"ottcher coordinate could be constructed, allowing the complex dilatation to vary in a neighbourhood of $z_0$. One would expect the complex dilatation would have to converge to some $\mu \in \D$ in a suitable sense near $z_0$ to be able to obtain such a result. See a result of Jiang \cite{Jiang} for the case where $f$ is itself asymptotically conformal in a neighbourhood of the fixed point.

\end{document}